\documentclass[11pt, final]{article}
\usepackage{a4}
\usepackage{amsmath}%
\usepackage{amstext}%
\usepackage{amssymb}%
\usepackage{showkeys}%
\usepackage{epsfig}%
\usepackage{cite}
\usepackage{tikz}
\usepackage{subfig}
\usepackage{caption}
\usepackage{multirow}
\usepackage{booktabs}
\usepackage{floatrow}

\usepackage{listings}
\usepackage{geometry}
\newtheorem{theorem}{Theorem}

\newtheorem{lemma}[theorem]{Lemma}

\newtheorem{rem}[theorem]{Remark}
\newenvironment{remark}{\rem\rm}{\endrem}

\newcounter{unnumber}

\newtheorem{prf}[unnumber]{Proof.}
\newenvironment{proof}{\prf\rm}{\hfill{$\blacksquare$}\endprf}
\newcommand{\R}{\mathbb{R}}%
\newcommand{\N}{\mathbb{N}}%
\DeclareMathOperator*\ran{ran}%
\DeclareMathOperator*\argmin{argmin}

\title{Second order dynamical systems associated to variational inequalities}

\author{Radu Ioan Bo\c{t} \thanks{University of Vienna, Faculty of Mathematics, Oskar-Morgenstern-Platz 1, A-1090 Vienna, Austria,
email: radu.bot@univie.ac.at.} \and
Ern\"{o} Robert Csetnek \thanks {University of Vienna, Faculty of Mathematics, Oskar-Morgenstern-Platz 1, A-1090 Vienna, Austria,
email: ernoe.robert.csetnek@univie.ac.at. Research supported by FWF (Austrian Science Fund), Lise Meitner Programme, project M 1682-N25.}}

\begin{document}
\maketitle

\noindent \textbf{Abstract.} We investigate the asymptotic convergence of the trajectories generated by the second order dynamical 
system $\ddot x(t) + \gamma\dot x(t) + \nabla \phi(x(t))+\beta(t)\nabla \psi(x(t))=0$, where $\phi,\psi:{\cal H}\rightarrow \R$ are 
convex and smooth functions defined on a real Hilbert space ${\cal H}$, $\gamma>0$ and $\beta$ is a function of time which controls 
the penalty term. We show weak convergence of the trajectories to a minimizer of the function $\phi$ over the (nonempty) set 
of minima of $\psi$ as well as convergence for the objective function values along the trajectories, provided a condition expressed 
via the Fenchel conjugate of $\psi$ is fulfilled. When the function $\phi$ is assumed to be strongly convex, we can even show 
strong convergence of the trajectories. The results can be seen as the second order counterparts of the ones given 
by Attouch and Czarnecki (Journal of Differential Equations 248(6), 1315--1344, 2010) for first order dynamical systems associated 
to constrained variational inequalities. At the same time we give a positive answer to an open problem posed in \cite{att-cza-16} by the 
same authors. \vspace{1ex}

\noindent \textbf{Key Words.} dynamical systems, Lyapunov analysis, convex optimization, 
nonautonomous systems\vspace{1ex}

\noindent \textbf{AMS subject classification.} 34G25, 47J25, 47H05, 90C25

\section{Introduction}\label{sec1}

Consider the optimization problem  
\begin{equation}\label{opt-intr}\inf_{x\in\argmin\psi}\phi(x),\end{equation}
where $\phi,\psi: {\cal H}\rightarrow \R\cup\{+\infty\}$ are proper, convex and lower semicontinuous functions defined on  a real Hilbert space ${\cal H}$ endowed with inner product $\langle\cdot,\cdot\rangle$
and associated norm $\|\cdot\|\!=\!\sqrt{\langle \cdot,\cdot\rangle}$. If $\partial \phi + N_{\argmin\psi}$ is maximally monotone, then determining an optimal solution $x \in {\cal H}$ of \eqref{opt-intr} means nothing else than solving the subdifferential inclusion problem
\begin{equation}\label{moninclusion}
\mbox{find} \ x \in {\cal H} \ \mbox{such that} \ 0 \in \partial \phi(x) +  N_{\argmin\psi}(x)
\end{equation}
or, equivalently, solving the variational inequality
\begin{equation}\label{vi}
\mbox{find} \ x \in {\cal H}  \ \mbox{and} \ p \in \partial \phi(x) \ \mbox{such that} \ \langle p, y-x \rangle \geq 0 \ \forall y \in \argmin\psi.
\end{equation}
Here, 
\begin{itemize}
\item $\partial \phi :  {\cal H} \rightrightarrows  {\cal H}$ is the convex subdifferential of $\phi$: $\partial \phi(x)=\{p\in {\cal H}: \phi(y)\geq \phi(x)+\langle p,y-x\rangle \ \forall y\in {\cal H}\}$ for $\phi(x) \in \R$ and  $\partial \phi(x)= \emptyset$ for $\phi(x) \not\in \R$;
\item $\argmin\psi$ denotes the set of minimizers of  $\psi$;
\item $N_{\argmin\psi}$ is the normal cone to the set $\argmin\psi$: $N_{\argmin\psi}(x)=\{p\in{\cal H}:\langle p,y-x\rangle\leq 0 \ \forall y\in \argmin\psi\}$ for  $x \in \argmin\psi$ and $N_{\argmin\psi}(x)=\emptyset$ for $x\not\in \argmin\psi$.
\end{itemize}
Attouch and Czarnecki investigated in \cite{att-cza-10} 
the asymptotic behavior of the trajectories of the nonautonomous first order dynamical system
\begin{equation}\label{1ord-att-cz}0\in\dot x(t)+\partial \phi(x(t))+\beta(t)\partial \psi(x(t))\end{equation}
to a solution of \eqref{moninclusion}, where $\beta:[0,+\infty)\rightarrow(0,+\infty)$ is a function of time assumed to tend to $+\infty$ as $t\rightarrow +\infty$. Several ergodic and nonergodic convergence results have been reported in \cite{att-cza-10} under the key assumption 
\begin{equation}\label{ass-att-cz}\forall p\in\ran N_{\argmin\psi} \ \int_0^{+\infty} 
\beta(t)\left[\psi^*\left(\frac{p}{\beta(t)}\right)-\sigma_{\argmin\psi}\left(\frac{p}{\beta(t)}\right)\right]dt<+\infty,
\end{equation}
where \begin{itemize}
\item $\psi^*: {\cal H}\rightarrow \R\cup\{+\infty\}$ is the Fenchel conjugate of $\psi$:
$\psi^*(p)=\sup_{x\in{\cal H}}\{\langle p,x\rangle-\psi(x)\} \ \forall p\in {\cal H};$ 
\item $\sigma_{\argmin\psi}: {\cal H}\rightarrow \R\cup\{+\infty\}$ is the support function of the set $\argmin\psi$:  
$\sigma_{\argmin\psi}(p)=\sup_{x\in {\argmin\psi}}\langle p,x\rangle$ for all $p\in {\cal H}$; 
\item $\ran N_{\argmin\psi}$ is the range of the normal cone $N_{\argmin\psi}$: that is $p\in\ran N_{\argmin\psi}$ if and only if there exists $x\in \argmin\psi$ such that $p\in N_{\argmin\psi}(x)$. 
\end{itemize}
We mention that $N_{\argmin\psi}=\partial \delta_{\argmin\psi}$, where $\delta_{\argmin\psi}:{\cal H}\rightarrow\R\cup\{+\infty\}$ is the indicator function 
of ${\argmin\psi}$, which takes the value $0$ on the set ${\argmin\psi}$ and $+\infty$, otherwise. Moreover, for $x\in {\argmin\psi}$ one has 
$p\in N_{\argmin\psi}(x)$ if and only if $\sigma_{\argmin\psi}(p)=\langle p,x\rangle$. 

Let us present a situation where the above condition \eqref{ass-att-cz} is fulfilled. According to \cite{att-cza-10}, if we take $\psi(x)=\frac{1}{2}\inf_{y\in C}\|x-y\|^2$, for a nonempty, convex and closed set $C \subseteq {\cal H}$, then the condition \eqref{ass-att-cz} is fulfilled if and only if 
$$\int_0^{+\infty}\frac{1}{\beta(t)} dt<+\infty,$$ which is trivially satisfied for $\beta(t)=(1+t)^\alpha$ with $\alpha>1$. 

The paper of Attouch and Czarnecki \cite{att-cza-10} was the starting point of a considerable number of research articles devoted to this subject, including those addressing generalizations to variational inequalities involving maximal monotone operators (see \cite{att-cza-10, att-cza-peyp-c, att-cza-peyp-p, noun-peyp, peyp-12, b-c-penalty-svva, b-c-penalty-vjm, banert-bot-pen, b-c-dyn-pen, att-cab-cz, att-maing}). In the literature enumerated above, the monotone inclusions problems have been approached
either through continuous dynamical systems or through their discrete counterparts formulated as splitting algorithms, both of penalty type. We speak in both cases about methods of penalty type, as the operator describing the underlying set of the addressed variational inequality is evaluated 
as a penalty functional. We refer also to the above-listed references for more general formulations
of the key assumption \eqref{ass-att-cz} and for further examples for which these conditions are satisfied. 

In this paper we are concerned with the asymptotic behavior of the second order dynamical system 
\begin{equation}\label{sec-or-intr}\ddot x(t) + \gamma\dot x(t) + \nabla \phi(x(t))+\beta(t)\nabla \psi(x(t))=0,\end{equation}
when the functions $\phi,\psi: {\cal H}\rightarrow \R$ are supposed to be convex and (Fr\'echet) differentiable with a Lipschitz continuous gradient. 

We show weak convergence of the trajectories to an optimal solution to \eqref{opt-intr} as well as 
convergence for the objective function values along the trajectories to the optimal objective value, 
provided \eqref{ass-att-cz} is fulfilled and the function $\beta$ satisfies some mild conditions. In case the 
function $\phi$ is strongly convex, the trajectories converge even strongly to the unique optimal solution of \eqref{opt-intr}. 

Let us also mention that the solving of \eqref{opt-intr} has been approached by means of second order dynamical system also in 
\cite[Section 3]{att-maing} by penalizing the gradient of the objective function with a vanishing coefficient. Strong asymptotic 
convergence of the trajectories to an optimal solution has been proved under the assumption that this gradient is strongly monotone 
(see also \cite{att-cza-02}). 

During the reviewing process of the initial version of our paper we became aware of a new preprint by Attouch and Czarnecki \cite{att-cza-16}, where a second order dynamical system connected to
\eqref{opt-intr}  is investigated, by penalizing the objective function of the latter through a nonincreasing 
function. On page 11 in \cite{att-cza-16} the authors leave as an open problem the approach of \eqref{opt-intr} via the second order 
dynamical system \eqref{sec-or-intr}, which this time penalizes the function describing the constraint set. 
The results presented in this paper  provide a positive answer to this question. 

\section{Preliminaries}\label{sec2}

In this section we present some preliminary results and tools that will be useful throughout the paper. 

The following results can be interpreted as continuous counterparts of the convergence statements for quasi-Fej\'er monotone sequences. For their proofs we refer the reader to 
\cite[Lemma 5.1]{abbas-att-sv} and \cite[Lemma 2.2 and 2.3]{alvarez2004}, respectively.

\begin{lemma}\label{fejer-cont1} Suppose that $F:[0,+\infty)\rightarrow\R$ is locally absolutely continuous and bounded from below and that
there exists $G\in L^1([0,+\infty))$ such that for almost every $t \in [0,+\infty)$ $$\dot F(t)\leq G(t).$$ 
Then there exists $\lim_{t\rightarrow \infty} F(t)\in\R$. 
\end{lemma}

\begin{lemma}\label{sec-ord-fej} Suppose that $F:[0,+\infty)\rightarrow\R$ is locally absolutely continuous
such that $F$ is bounded from below, $\dot F$ is locally absolutely continuous and there exist $\gamma >0$ and $G\in L^1([0,+\infty))$ fulfilling for almost every 
$t \in [0,+\infty)$ $$\ddot F(t)+\gamma\dot F(t)\leq G(t).$$ 
Then there exists $\lim_{t\rightarrow \infty} F(t)\in\R$. 
\end{lemma}

We are interested in the asymptotic analysis of the following second order nonautonomous dynamical system: 
\begin{equation}\label{dyn-syst}\left\{
\begin{array}{ll}
\ddot x(t) + \gamma\dot x(t) + \nabla \phi(x(t))+\beta(t)\nabla \psi(x(t))=0\\
x(0)=u_0, \dot x(0)=v_0,
\end{array}\right.\end{equation}
where ${\cal H}$ is a real Hilbert space, $\gamma>0$, $u_0,v_0\in {\cal H}$, 
$\beta:[0,+\infty)\rightarrow (0,+\infty)$ is a locally integrable function and the following conditions hold:
\begin{align*}(H_\psi)& \ \psi:{\cal H}\rightarrow [0,+\infty) \mbox{ is convex, (Fr\'{e}chet) differentiable}
 \mbox{ with Lipschitz continuous gradient}\\& \mbox{ and }\argmin\psi=\psi^{-1}(0)\neq\emptyset;\\
(H_\phi)& \ \phi:{\cal H}\rightarrow \R \mbox{ is convex, (Fr\'{e}chet) differentiable with Lipschitz continuous gradient,}\\
& \mbox{ bounded from below and }S:=\{z\in \argmin\psi: \phi(z)\leq \phi(x) \  \forall x\in\argmin\psi\}\neq\emptyset. \end{align*}

We call $x:[0,+\infty)\rightarrow{\cal H}$ (strong global) solution of \eqref{dyn-syst} if $x$ and $\dot x$ 
are locally absolutely continuous (in other words, absolutely continuous on each interval $[0,b]$ for $0<b<+\infty$), $x(0)=u_0, \dot x(0)=v_0$ and $\ddot x(t) + \gamma\dot x(t) + \nabla \phi(x(t))+\beta(t)\nabla \psi(x(t))=0$ for almost every $t \in [0,+\infty)$.

Due to the Lipschitz continuity of $\nabla \psi$ and $\nabla \phi$ assumed in $(H_\psi)$ and $(H_\phi)$, respectively, the existence 
and uniqueness of (strong global) solutions of \eqref{dyn-syst} is a consequence of the Cauchy-Lipschitz-Picard Theorem 
(see for example \cite{att-maing, b-c-dyn-sec-ord, alv-att-bolte-red, cabot-engler-gadat-tams}). 

In the following sections we provide sufficient conditions which guarantee both ergodic and nonergodic convergence  of the trajectory of \eqref{dyn-syst} to an optimal solution of \eqref{opt-intr}.

\begin{remark}\label{rem-heavy-ball} \begin{enumerate}\item[(a)] In case $\psi=0$ the dynamical system \eqref{dyn-syst} becomes the second order gradient system 
associated to the the heavy ball method (see also \cite{alvarez2000, att-alv}): 
\begin{equation}\label{heavy-ball}\left\{
\begin{array}{ll}
\ddot x(t) + \gamma\dot x(t) + \nabla \phi(x(t))=0\\
x(0)=u_0, \dot x(0)=v_0.
\end{array}\right.\end{equation}
In this case $\argmin \psi={\cal H}$ and  \eqref{opt-intr} is nothing else than the minimization of the function $\phi$ over ${\cal H}$.

\item[(b)] It is well known that time discretization of second order dynamical systems leads to iterative algorithms involving inertial terms, which 
basically means that  the new iterate is defined by making use of the previous two iterates (see for example \cite{alvarez2004, alvarez-attouch2001}). 
We refer the reader to  \cite{b-c-penalty-inertial} where an inertial algorithm of penalty type has been proposed and investigated from the point of view of its convergence properties.
\end{enumerate}
\end{remark}

\section{Convergence of the trajectories and of the objective function values}\label{sec3}

In this section we show that we have weak convergence for the trajectory generated by the dynamical system \eqref{dyn-syst} to an optimal solution 
of \eqref{opt-intr} as well as convergence for the objective function values along the trajectory. 
To this end, we make use of the following supplementary assumptions: 
\begin{enumerate}
\item[($H_{\beta}$)] $\beta :[0,+\infty) \rightarrow (0,+\infty)$ is a $C^1$-function and it satisfies the growth condition $0\leq\dot\beta\leq k\beta$, where $0 \leq k<\gamma$;
 \item[$(H)$] $\forall p\in\ran N_{\argmin \psi} \ \int_0^{+\infty} 
\beta(t)\left[\psi^*\left(\frac{p}{\beta(t)}\right)-\sigma_{\argmin \psi}\left(\frac{p}{\beta(t)}\right)\right]dt<+\infty$. 
\end{enumerate}

\begin{remark}
\begin{itemize} 
\item[(a)] Under $(H_\psi)$, due to $\psi\leq\delta_{\argmin \psi}$, we 
have $$\psi^*\geq\delta_{\argmin \psi}^*=\sigma_{\argmin \psi}.$$
\item[(b)] When $\psi=0$ (see Remark \ref{rem-heavy-ball}(a)), it holds $N_{\argmin \psi}(x)=\{0\}$ for every $x\in \argmin \psi={\cal H}$, 
$\psi^*=\sigma_{\argmin \psi}=\delta_{\{0\}}$ and $(H)$ trivially holds. 
\end{itemize}
\end{remark}

We start with the following technical result.

\begin{lemma}\label{l1-op} Assume that $(H_\psi)$, $(H_\phi)$, 
$(H_{\beta})$ and $(H)$ hold and let $x:[0,+\infty)\rightarrow{\cal H}$ be the trajectory generated by the dynamical system \eqref{dyn-syst}. Then for every $z\in S$ the following statements are true:
\begin{enumerate}
 \item[(i)] $x$ and $\dot x$ are bounded;  
 \item[(ii)] $\int_0^{+\infty}\beta(t)\psi(x(t))dt<+\infty$; 
 \item[(iii)] $\exists\lim_{t\rightarrow+\infty}\int_0^t\langle \nabla\phi(z),x(s)-z\rangle ds\in\R$;
 \item[(iv)] $\exists\lim_{t\rightarrow+\infty}\int_0^t\left(\phi(x(s))-\phi(z)+\left(1-\frac{k}{\gamma}\right)\beta(s)\psi(x(s))\right)ds\in\R$; 
 \item[(v)] $\dot x\in L^2([0,+\infty);{\cal H})$; 
 \item[(vi)] $\exists\lim_{t\rightarrow+\infty}\|x(t)-z\|\in\R$. 
\end{enumerate}
\end{lemma}

\begin{proof} Take an arbitrary $z\in S$ and define $$h_z(t)=\frac{1}{2}\|x(t)-z\|^2$$
and the energy functional \begin{equation}\label{E}E(t)=\frac{1}{2}\|\dot x(t)\|^2+\phi(x(t))+\beta(t)\psi(x(t)).\end{equation}
It immediately follows that for almost every $t \in [0,+\infty)$ 
\begin{equation}\label{dh}\dot h_z(t)=\langle \dot x(t),x(t)-z\rangle,\end{equation} 
\begin{equation}\label{ddh}\ddot h_z(t)=\|\dot x(t)\|^2+\langle\ddot x(t),x(t)-z\rangle,\end{equation}
hence, due to \eqref{dyn-syst}
\begin{equation}\label{h2-h1} \ddot h_z(t)+\gamma\dot h_z(t)=\|\dot x(t)\|^2+\langle -\nabla \phi(x(t))-\beta(t)\nabla\psi(x(t)),x(t)-z\rangle       .
\end{equation}
Further, from \eqref{dyn-syst} we obtain  for almost every $t \in [0,+\infty)$ 
\begin{align}\label{e1}\dot E(t)& =\langle \ddot x(t),\dot x(t)\rangle+\langle \dot x(t),\nabla\phi(x(t))\rangle+
\beta(t)\langle \dot x(t),\nabla\psi(x(t))\rangle+\dot \beta(t)\psi(x(t))\nonumber\\
& =\langle\dot x(t),\ddot x(t)+\nabla\phi(x(t))+\beta(t)\nabla\psi(x(t))\rangle+\dot \beta(t)\psi(x(t))\nonumber\\
& = -\gamma\|\dot x(t)\|^2+\dot \beta(t)\psi(x(t)).\end{align}
By combining \eqref{h2-h1} and \eqref{e1} we derive for almost every $t \in [0,+\infty)$ 
\begin{equation}\label{h-e} \ddot h_z(t)+\gamma\dot h_z(t)+\frac{1}{\gamma}\dot E(t)-\frac{\dot \beta(t)}{\gamma}\psi(x(t))= 
\langle -\nabla \phi(x(t))-\beta(t)\nabla\psi(x(t)),x(t)-z\rangle. 
\end{equation}          
The convexity of the functions $\phi$ and $\psi$, the fact that $z\in \argmin\psi$ (hence $\psi(z)=0$) and the non-negativity of $\beta$ and 
$\psi$ yield 
$$\langle -\nabla \phi(x(t))-\beta(t)\nabla\psi(x(t)),x(t)-z\rangle\leq \phi(z)-\phi(x(t))-\beta(t)\psi(x(t)).$$
This inequality in combination with \eqref{h-e} implies 
\begin{equation}\label{h-e-1} \ddot h_z(t)+\gamma\dot h_z(t)+\frac{1}{\gamma}\dot E(t)-\frac{\dot \beta(t)}{\gamma}\psi(x(t)) 
\leq\phi(z)-\phi(x(t))-\beta(t)\psi(x(t))
\end{equation} 
for almost every $t \in [0,+\infty)$. By using now the growth condition for $\beta$ we get 
\begin{equation}\label{h-e-2} \ddot h_z(t)+\gamma\dot h_z(t)+\frac{1}{\gamma}\dot E(t) 
+\phi(x(t))-\phi(z)+\widetilde\beta(t)\psi(x(t))\leq 0 \ \mbox{for almost every} \ t \in [0,+\infty), 
\end{equation} 
where \begin{equation}\label{def-beta-tilde}\widetilde\beta(t):=\left(1-\frac{k}{\gamma}\right)\beta(t).\end{equation}
Since $z$ is an optimal solution for \eqref{opt-intr}, the first order optimality condition 
delivers us \begin{equation}\label{zero-in-partial}0\in\partial (\phi+\delta_{\argmin \psi})(z)=\nabla\phi(z)+ N_{\argmin \psi}(z),\end{equation}
hence \begin{equation}\label {opt-cond} -\nabla \phi(z)\in N_{\argmin \psi}(z) \subseteq \ran N_{\argmin \psi}.\end{equation}

By using this fact and the Young-Fenchel inequality we obtain for every $t \in [0,+\infty)$
\begin{align}\label{ineq-conj}\widetilde\beta(t)\psi(x(t))+\langle \nabla \phi(z),x(t)-z\rangle & =
\widetilde\beta(t)\left(\psi(x(t))+\left\langle \frac{\nabla \phi(z)}{\widetilde\beta(t)},x(t)-z\right\rangle\right)\nonumber\\
& = \widetilde\beta(t)\left(\psi(x(t))-\left\langle \frac{-\nabla \phi(z)}{\widetilde\beta(t)},x(t)\right\rangle+\sigma_{\argmin \psi}\left(\frac{-\nabla \phi(z)}{\widetilde\beta(t)}\right)\right)\nonumber\\
& \geq \widetilde\beta(t)\left(-\psi^*\left(\frac{-\nabla \phi(z)}{\widetilde\beta(t)}\right)+\sigma_{\argmin \psi}\left(\frac{-\nabla \phi(z)}{\widetilde\beta(t)}\right)\right).\end{align}

Furthermore, \begin{equation}\label{phi-conv}\phi(x(t))-\phi(z)\geq \langle \nabla \phi(z),x(t)-z\rangle \ \forall t \in [0,+\infty).\end{equation}

Finally, from \eqref{h-e-2}, \eqref{ineq-conj} and \eqref{phi-conv} we obtain for almost every $t \in [0,+\infty)$
\begin{align} & \ddot h_z(t)+\gamma\dot h_z(t)+\frac{1}{\gamma}\dot E(t)+
\widetilde\beta(t)\left(-\psi^*\left(\frac{-\nabla\phi(z)}{\widetilde\beta(t)}\right)+\sigma_{\argmin \psi}\left(\frac{-\nabla\phi(z)}{\widetilde\beta(t)}\right)\right)\nonumber\\
\leq & \ \ddot h_z(t)+\gamma\dot h_z(t)+\frac{1}{\gamma}\dot E(t) +\widetilde\beta(t)\psi(x(t))+\langle \nabla\phi(z),x(t)-z\rangle\nonumber\\
\leq &\ \ddot h_z(t)+\gamma\dot h_z(t)+\frac{1}{\gamma}\dot E(t) 
+\phi(x(t))-\phi(z)+\widetilde\beta(t)\psi(x(t))\nonumber\\\leq & \ 0.\label{ineq2'}
\end{align}

Thus 
\begin{equation}\label{ineq3}\ddot h_z(t)+\gamma\dot h_z(t)\leq -\frac{1}{\gamma}\dot E(t)+
\widetilde\beta(t)\left(\psi^*\left(\frac{-\nabla\phi(z)}{\widetilde\beta(t)}\right)-\sigma_{\argmin \psi}\left(\frac{-\nabla\phi(z)}{\widetilde\beta(t)}\right)\right)\end{equation}
 for almost every $t \in [0,+\infty)$.
 
(i) From \eqref{ineq2'} we have for almost every $t \in [0,+\infty)$
\begin{align*}
 \ddot h_z(t)+\gamma\dot h_z(t)+\frac{1}{\gamma}\dot E(t)\leq
\widetilde\beta(t)\left(\psi^*\left(\frac{-\nabla\phi(z)}{\widetilde\beta(t)}\right)-\sigma_{\argmin \psi}\left(\frac{-\nabla\phi(z)}{\widetilde\beta(t)}\right)\right).
\end{align*}
By integrating this inequality from $0$ to $T$ ($T>0$) and by taking into account $(H)$, \eqref{E} and the fact that $\phi$ and $\psi$ 
are bounded from below, it yields that there exists $M>0$ such that  
\begin{equation}\label{ineq4}\dot h_z(T)+\gamma h_z(T)+\frac{1}{2\gamma}\|\dot x(T)\|^2\leq M \ \forall T\geq 0.\end{equation}
We derive $$\dot h_z(T)+\gamma h_z(T)\leq M \ \forall T\geq 0.$$
By multiplying this inequality with $\exp(\gamma T)$ and then integrating from $0$ to $s$, where $s > 0$, one easily obtains
$$h_z(s)\leq h_z(0)\exp(-\gamma s)+\frac{M}{\gamma}(1-\exp(-\gamma s)) \ \forall s> 0.$$
Thus $h_z$ is bounded, hence $x$ is bounded. 

Further, \eqref{ineq4} implies $$\dot h_z(T)+\frac{1}{2\gamma}\|\dot x(T)\|^2\leq M \ \forall T\geq 0.$$

Since $t\mapsto x(t)$ is bounded, we deduce from the above inequality and \eqref{dh} that $\dot x$ is bounded as well. 

(ii) Let $F:[0,+\infty)\rightarrow\R$ be defined by 
$$F(t)=\int_0^t\left(-\widetilde\beta(s)\psi(x(s))-\langle\nabla\phi(z),x(s)-z\rangle\right)ds \ \forall t\geq 0.$$

Making again use of \eqref{ineq2'} we have for almost every $t \in [0,+\infty)$ 
$$\ddot h_z(t)+\gamma\dot h_z(t)+\frac{1}{\gamma}\dot E(t)\leq-\widetilde\beta(t)\psi(x(t))-\langle \nabla \phi(z),x(t)-z\rangle.$$
Since $h_z$ and $E$ are bounded from below and $\dot h_z$ is bounded, by integrating the last inequality we derive that 
$F$ is bounded from below. Moreover, from \eqref{ineq-conj} we have for almost every $t \in [0,+\infty)$ 
$$\dot F(t)\leq \widetilde\beta(t)\left(\psi^*\left(\frac{-\nabla\phi(z)}{\widetilde\beta(t)}\right)
-\sigma_{\argmin \psi}\left(\frac{-\nabla\phi(z)}{\widetilde\beta(t)}\right)\right).$$ 
Since according to (H), the function on the right-hand side of this inequality is $L^1$-integrable on $[0,+\infty)$, a direct application of 
Lemma \ref{fejer-cont1} yields that $\lim_{t\rightarrow+\infty}F(t)$ exists and is a real number. Hence 
\begin{equation}\label{b}\exists \lim_{t\rightarrow+\infty}\int_0^t\left(\widetilde\beta(s)\psi(s)+\langle \nabla \phi(z),x(s)-z\rangle\right)ds\in\R\end{equation}

Since $\psi\geq 0$, we obtain for every $t \in [0,+\infty)$ 
$$\widetilde\beta(t)\psi(x(t))+\langle \nabla \phi(z),x(t)-z\rangle\geq \frac{\widetilde\beta(t)}{2}\psi(x(t))+\langle \nabla \phi(z),x(t)-z\rangle$$ and from here,
similarly to \eqref{ineq-conj},
$$\frac{\widetilde\beta(t)}{2}\psi(x(t))+\langle \nabla \phi(z),x(t)-z\rangle\geq 
\frac{\widetilde\beta(t)}{2}\left(-\psi^*\left(\frac{-2\nabla \phi(z)}{\widetilde\beta(t)}\right)+\sigma_{\argmin \psi}\left(\frac{-2\nabla \phi(z)}{\widetilde\beta(t)}\right)\right).$$

Thus, for almost every $t \in [0,+\infty)$ it holds
\begin{align*} & \ddot h_z(t)+\gamma\dot h_z(t)+\frac{1}{\gamma}\dot E(t)+
\frac{\widetilde\beta(t)}{2}\left(-\psi^*\left(\frac{-2\nabla \phi(z)}{\widetilde\beta(t)}\right)+\sigma_{\argmin \psi}\left(\frac{-2\nabla \phi(z)}{\widetilde\beta(t)}\right)\right)\\
\leq & \ \ddot h_z(t)+\gamma\dot h_z(t)+\frac{1}{\gamma}\dot E(t) +\frac{\widetilde\beta(t)}{2}\psi(x(t))+\langle \nabla \phi(z),x(t)-z\rangle\\
\leq & \ \ddot h_z(t)+\gamma\dot h_z(t)+\frac{1}{\gamma}\dot E(t) +\widetilde\beta(t)\psi(x(t))+\langle \nabla \phi(z),x(t)-z\rangle\\
\leq & \ 0.
\end{align*}
By using the same arguments as in the proof of \eqref{b} it yields that
\begin{equation}\label{b2}\exists \lim_{t\rightarrow+\infty}\int_0^t\left(\frac{\widetilde\beta(s)}{2}\psi(s)+\langle \nabla \phi(z),x(s)-z\rangle\right)ds\in\R.\end{equation}
Finally, from \eqref{b}, \eqref{b2} and \eqref{def-beta-tilde} we obtain (ii). 

(iii) Follows from (ii), \eqref{b} and \eqref{def-beta-tilde}.

(iv) Follows from \eqref{ineq2'} and \eqref{ineq-conj}, by using the same arguments as for proving statement \eqref{b}. 

(v) From \eqref{e1} and $(H_{\beta})$ we derive for almost every $t\geq 0$ 
\begin{equation}\label{ineq7}\dot E(t)+\gamma\|\dot x(t)\|^2\leq k \beta(t)\psi(x(t)).\end{equation}
Since $E$ is bounded from below, integrating the last inequality and taking into account (ii) we conclude that $\dot x\in L^2([0,+\infty);{\cal H})$. 

(vi) Combining \eqref{ineq3} with \eqref{e1} we obtain 
\begin{equation}\label{ineq5}\ddot h_z(t)+\gamma\dot h_z(t)+\frac{\dot \beta(t)}{\gamma}\psi(x(t))\leq \|\dot x(t)\|^2+
\widetilde\beta(t)\left(\psi^*\left(\frac{-\nabla\phi(z)}{\widetilde\beta(t)}\right)-\sigma_{\argmin \psi}\left(\frac{-\nabla\phi(z)}{\widetilde\beta(t)}\right)\right)\end{equation}
 for almost every $t \in [0,+\infty)$. Notice that $\dot\beta(t)\geq 0$, hence 
 \begin{equation}\label{ineq6}\ddot h_z(t)+\gamma\dot h_z(t)\leq \|\dot x(t)\|^2+
\widetilde\beta(t)\left(\psi^*\left(\frac{-\nabla\phi(z)}{\widetilde\beta(t)}\right)-\sigma_{\argmin \psi}\left(\frac{-\nabla\phi(z)}{\widetilde\beta(t)}\right)\right)\end{equation}
 for almost every $t \in [0,+\infty)$.

According to (v), $(H)$, \eqref{opt-cond}, \eqref{def-beta-tilde} and the fact that $\ran N_{\argmin \psi}$ is a cone, 
the function on the right-hand side of the above inequality is $L^1$-integrable on $[0,+\infty)$. 
Applying now Lemma \ref{sec-ord-fej} we deduce that (i) holds. 
\end{proof}

\begin{remark} The assumption $\dot\beta\geq 0$ has been used in the above proof only for showing statement (vi).   
\end{remark}

For the main result of the paper concerning the weak asymptotic convergence of the trajectory generated by the dynamical 
system \eqref{dyn-syst}, that we formulate and prove afterwards, we will make use of the continuous version of the Opial Lemma.

\begin{lemma}\label{opial-cont} Let $S$ be a nonempty subset of the real Hilbert space ${\cal H}$ and $x:[0,+\infty)\rightarrow{\cal H}$ 
a given function. Assume that 
\begin{enumerate}
\item [(i)] $\lim_{t\rightarrow+\infty}\|x(t)-z\|$ exists for every $z \in S$; 
\item[(ii)] every weak limit point of $x$ belongs to $S$. 
\end{enumerate}
Then there exists $x_{\infty}\in S$ such that $x(t)$ converges weakly to $x_{\infty}$ as $t\rightarrow+\infty$. 
\end{lemma}

We can state now the main theorem of the paper. 

\begin{theorem}\label{th-nonerg-conv} Assume that $(H_\psi)$, $(H_\phi)$, 
$(H_{\beta})$ and $(H)$ hold under the supplementary condition $\lim_{t\rightarrow+\infty}\beta(t)=+\infty$ and let 
$x:[0,+\infty)\rightarrow{\cal H}$ be the trajectory generated by the dynamical system \eqref{dyn-syst}. 
Then the following statements are true: 
\begin{enumerate}
 \item [(i)] $\phi(x(t))$ converges to the optimal objective value of \eqref{opt-intr} as $t\rightarrow+\infty$;
 \item[(ii)] $\lim_{t\rightarrow+\infty}\beta(t)\psi(x(t))=\lim_{t\rightarrow+\infty}\psi(x(t))=0$;
 \item[(iii)] $\int_0^{+\infty}\beta(t)\psi(x(t))dt<+\infty$;
 \item[(iv)] $\int_0^{+\infty}\|\dot x(t)\|^2dt<+\infty$; 
 \item [(v)] $\lim_{t\rightarrow+\infty}\dot x(t)=0$; 
 \item[(vi)] there exists $x_{\infty}\in S$ such that $x(t)$ converges weakly to $x_{\infty}$ as $t\rightarrow+\infty$. 
\end{enumerate} 
\end{theorem}

\begin{proof} Consider again the energy functional defined in \eqref{E}. Taking into account that $E$ is bounded from below, from 
\eqref{ineq7}, Lemma \ref{l1-op}(ii) and Lemma \ref{fejer-cont1} we derive that 
\begin{equation}\label{exists-lim-e}\exists\lim_{t\rightarrow+\infty}E(t)\in\R.\end{equation}
We claim that \begin{equation}\label{psi-0}\lim_{t\rightarrow+\infty}\psi(x(t))=0.\end{equation}
Indeed, notice that 
\begin{equation}\label{e/b}\frac{E(t)}{\beta(t)}=\frac{\|\dot x(t)\|^2}{2\beta(t)}+\frac{\phi(x(t))}{\beta(t)}+\psi(x(t)).\end{equation}
Let $z\in S$ be arbitrary. Since $\phi$ is bounded from below and by taking into account the inequality 
$$\phi(x(t))\leq \phi(z)+\langle \nabla\phi(x(t)),x(t)-z\rangle \forall t \in [0,+\infty),$$
that $x$ is bounded, $\nabla \phi$ is Lipschitz continuous and $\lim_{t\rightarrow+\infty}\beta(t)=+\infty$, we can easily conclude that
$\lim_{t\rightarrow+\infty}\frac{\phi(x(t))}{\beta(t)}=0$. Furthermore, since $\dot x$ is bounded, by using \eqref{exists-lim-e}, 
we easily derive \eqref{psi-0} by passing to the limit as $t\rightarrow+\infty$ in \eqref{e/b}. 

Now we invoke \cite[Lemma 3.4]{att-cza-10} for the relation 
\begin{equation}\label{liminf-phi}\liminf_{t\rightarrow+\infty}\phi(x(t))\geq\phi(z).\end{equation} Indeed, 
this was stated in \cite[Lemma 3.4]{att-cza-10} for the first order companion of the dynamical system \eqref{dyn-syst}. 
The result remains true for \eqref{dyn-syst}, too, since a careful look at the proof of Lemma 3.4 in \cite{att-cza-10} reveals 
that ingredients which lead to the conclusion are the statements in Lemma \ref{l1-op}(iii), the fact that the trajectory $x$ is bounded, 
the weak lower semicontinuity of $\psi$, equation \eqref{psi-0}, the inequality \eqref{phi-conv} and the relation \eqref{opt-cond}.

Since $E(t)\geq \phi(x(t))$, from \eqref{liminf-phi} we have $\lim_{t\rightarrow+\infty}E(t)\geq \phi(z)$. 
We claim that \begin{equation}\label{lim-e}\lim_{t\rightarrow+\infty}E(t)= \phi(z).\end{equation}

Let us assume that $\lim_{t\rightarrow+\infty}E(t) > \phi(z)$. Then there exist $\theta>0$ and $t_0\geq 0$ such that 
for every $t\geq t_0$ we have 
\begin{equation}\phi(x(t))+\beta(t)\psi(x(t))+\frac{1}{2}\|\dot x(t)\|^2>\phi(z)+\theta.\end{equation}
Hence for every $t\geq t_0$ 
\begin{equation}\theta<\phi(x(t))-\phi(z)+\left(1-\frac{k}{\gamma}\right)\beta(t)\psi(x(t))+
\frac{k}{\gamma}\beta(t)\psi(x(t))+\frac{1}{2}\|\dot x(t)\|^2.\end{equation}
Integrating the last inequality and taking into account Lemma \ref{l1-op}(ii), (iv) and (v) we obtain a contradiction. Hence 
\eqref{lim-e} holds. 

Since $\phi(x(t))\leq E(t)$ we derive from \eqref{lim-e} that $\limsup_{t\rightarrow+\infty}\phi(x(t))\leq\phi(z)$, which combined 
with \eqref{liminf-phi} imply (i). 

Further, notice that due to $\lim_{t\rightarrow+\infty}E(t)=\lim_{t\rightarrow+\infty}\phi(x(t))=\phi(z)$, from 
\eqref{E} we obtain $$\lim_{t\rightarrow+\infty}\left(\beta(t)\psi(x(t))+\frac{1}{2}\|\dot x(t)\|^2\right)=0,$$
hence \begin{equation}\label{b-psi}\lim_{t\rightarrow+\infty}\beta(t)\psi(x(t))=0\end{equation}
and $$\lim_{t\rightarrow+\infty}\|\dot x(t)\|=0.$$

Thus (ii) and (v) hold. The statements (iii) and (iv) have been proved in Lemma \ref{l1-op}. 

Statement (vi) will be a consequence of the Opial Lemma. Since according to  Lemma \ref{l1-op}(vi), the first condition 
in the Opial Lemma is fulfilled, it remains to show that
every weak limit point of $x(t)$ as $t\rightarrow+\infty$ belongs to $S$. 

Let $(t_n)_{n\in\N}$ be a sequence of positive numbers such that $\lim_{n\rightarrow+\infty}t_n=+\infty$ and $x(t_n)$ converges weakly 
to $x_{\infty}$ as $n\rightarrow+\infty$. By using the weak lower semicontinuity of 
$\psi$ and (ii) we obtain 
$$0\leq \psi(x_{\infty})\leq\liminf_{n\rightarrow+\infty}\psi(x(t_n))=0,$$
hence $x_{\infty}\in \argmin\psi$. Moreover, the weak lower semicontinuity of 
$\phi$ and (i) yield
$$\phi(x_{\infty})\leq\liminf_{n\rightarrow+\infty}\phi(x(t_n))=\phi(z),$$
thus $x_{\infty}\in S$.  
\end{proof}

We show in the following that if the objective function of \eqref{opt-intr} is strongly convex, then the trajectory generated by \eqref{dyn-syst} converges strongly to the unique optimal solution of \eqref{opt-intr}. 

\begin{theorem}\label{str-mon} Assume that $(H_\psi)$, $(H_\phi)$, 
$(H_{\beta})$ and $(H)$ hold and let $x:[0,+\infty)\rightarrow{\cal H}$ be the trajectory generated by the dynamical system \eqref{dyn-syst}.  If $\phi$ is strongly convex, then $x(t)$ converges strongly to the unique optimal solution of \eqref{opt-intr}  as $t\rightarrow+\infty$. 
\end{theorem}

\begin{proof} Let $\gamma>0$ be such that $\phi$ is $\gamma$-strongly 
convex. It is a well-known fact that in case the optimization problem 
\eqref{opt-intr} has a unique optimal solution, which we denote by $z$. 

By combining \eqref{ineq-conj} with the stronger inequality
\begin{equation}\label{phi-str-conv}\phi(x(t))-\phi(z)\geq \langle \nabla \phi(z),x(t)-z\rangle+\frac{\gamma}{2}\|x(t)-z\|^2 \ \forall t \in [0,+\infty),\end{equation}
we obtain this time (see the proof of Lemma \ref{l1-op}) for almost every $t \in [0,+\infty)$
\begin{align*} & \ddot h_z(t)+\gamma\dot h_z(t)+\frac{\gamma}{2}\|x(t)-z\|^2 + \frac{1}{\gamma}\dot E(t)+
\widetilde\beta(t)\left(-\psi^*\left(\frac{-\nabla\phi(z)}{\widetilde\beta(t)}\right)+\sigma_{\argmin \psi}\left(\frac{-\nabla\phi(z)}{\widetilde\beta(t)}\right)\right)\nonumber\\
\leq & \ \ddot h_z(t)+\gamma\dot h_z(t)+\frac{\gamma}{2}\|x(t)-z\|^2 + \frac{1}{\gamma}\dot E(t) +\widetilde\beta(t)\psi(x(t))+\langle \nabla\phi(z),x(t)-z\rangle\nonumber\\
\leq &\ \ddot h_z(t)+\gamma\dot h_z(t)+\frac{1}{\gamma}\dot E(t) 
+\phi(x(t))-\phi(z)+\widetilde\beta(t)\psi(x(t))\nonumber\\\leq & \ 0.
\end{align*}

Taking into account $(H)$ and that $E$ and $h_z$ are bounded from below, by integration of the above inequality we obtain that there exists a constant $L\in\R$ such that $$\dot h_z(T)+\frac{\gamma}{2}\int_0^T\|x(t)-z\|^2dt\leq L \ \forall T>0.$$
Notice that due to Lemma \ref{l1-op}(i) the function $T\mapsto\dot h_z(T)$ is bounded, hence 
$$\int_0^{+\infty}\|x(t)-z\|^2dt<+\infty.$$ Since according to Lemma \ref{l1-op}(vi) $\lim_{t\rightarrow+\infty}\|x(t)-z\|$ exists, we conclude that 
$\|x(t)-z\|$ converges to $0$ as $t\rightarrow+\infty$ and the proof is complete.  
\end{proof}

\begin{remark}\label{rem13} 
The results presented in this paper remain true even if the assumed growth condition is satisfied starting
with a $t_0\geq 0$, that is, if there exists $t_0\geq 0$ such that 
$$0\leq\dot\beta(t)\leq k\beta(t) \ \forall t\geq t_0.$$
\end{remark}

\noindent {\bf Acknowledgements.} The authors are grateful to an anonymous reviewer for pointing out a flaw in the initial version 
of the paper.

\end{document}